\documentclass[10pt]{article}
\usepackage[utf8]{inputenc}
\usepackage{amsmath}
\usepackage{amssymb}
\usepackage{mathrsfs}
\usepackage{amsfonts}
\usepackage{mathrsfs,amscd,amssymb,amsthm,amsmath,bm,graphicx}
\usepackage{graphicx}
\usepackage{cite}
\usepackage{color}
\usepackage{enumitem}
\makeatletter

\renewcommand{\@seccntformat}[1]{{\csname the#1\endcsname}{\normalsize.}\hspace{.5em}}
\makeatother

\def \[{\begin{equation}}
\def \]{\end{equation}}

\newtheorem{thm}{Theorem}[section]
\newtheorem{defi}{Definition}
\newtheorem{lem}[thm]{Lemma}

\newtheorem{cor}[thm]{Corollary}
\newtheorem{prop}[thm]{Proposition}

\begin{document}
\setlength{\baselineskip}{13pt}
\begin{center}{\Large \bf The statistical analysis for Sombor indices in a random polygonal chain networks
}

\vspace{4mm}

{\large Jia-Bao Liu$^{1,*}$, Ya-Qian Zheng$^{1,*}$, Xin-Bei Peng$^1$\vspace{2mm}

{\small $^1$School of Mathematics and Physics, Anhui Jianzhu
University, Hefei 230601, P.R. China}
\vspace{2mm}
}\end{center}
\footnotetext{E-mail address: liujiabaoad@163.com,
zhengyaqian168@163.com, pengxinbeiajd@163.com.}
\footnotetext{* Corresponding author.}

{\noindent{\bf Abstract.}\ \ The Sombor indices, a new category of degree-based topological molecular descriptors, have been widely investigated due to their excellent chemical applicability. This paper aims to establish Sombor indices distributions in random polygonal chain networks and to achieve expressions of the expected values and variances. The expected values and variances of the Sombor indices for polyonino, pentachain, polyphenyl, and cyclooctane chains are obtained.
 Since the end connection of a random chain network follows a binomial distribution, the Sombor indices of any chain network follow the normal distribution when the number of polygons connected by the chain, indicated by $n$, approaches infinity.
\vspace{2mm}

\section{Introduction}
\ \ \ \ Let $\mathbb{G}=\left(V\left(\mathbb{G}\right),\ E\left(\mathbb{G} \right)\right)$ denote a graph with the edge set $E\left(\mathbb{G} \right)$ and the vertex set $V\left(\mathbb{G} \right)$.
The degree of vertex $u$ is represented by $d_{\mathbb{G} }\left(u\right)$. If $u,v\in  V\left(\mathbb{G} \right)$ are adjacent, the edge connecting them is labeled by $uv$. Please refer to\cite{BJA} for concepts or notations in graph theory that are not addressed in this paper. A multitude of degree-based topological indices play a very significant role in the fields of mathematics and chemistry. 
For further information on topological indices, refer to\cite{PMS}. Gutman\cite{GI} has proposed a new group of topological indices called the Sombor indices, which includes the (ordinary) Sombor index, the reduced Sombor index, and the average Sombor index, inspired by the Euclidean metric.

For a graph $\mathbb{G} $, the formulas of the Sombor index $SO\left(\mathbb{G} \right)$, the reduced Sombor index $SO_{red}\left(\mathbb{G} \right)$ and the average Sombor index $SO_{avr}\left(\mathbb{G} \right)$ are given as follows:

\begin{eqnarray*}
    SO\left(\mathbb{G} \right)=\sum_{uv\in E\left(\mathbb{G} \right)}\sqrt{d^2_\mathbb{G} \left(u\right)+d^2_\mathbb{G} \left(v\right)}\hfill \ \  ,
   \end{eqnarray*}
 \begin{eqnarray*}
    SO_{red}\left(\mathbb{G} \right)=\sum_{uv\in E\left(\mathbb{G} \right)}\sqrt{\left(d_\mathbb{G} \left(u\right)-1\right)^2+\left(d_\mathbb{G} \left(v\right)-1\right)^2}\hfill \ \ ,
   \end{eqnarray*}
\begin{eqnarray*}
    SO_{red}\left(\mathbb{G} \right)=\sum_{uv\in E\left(\mathbb{G} \right)}\sqrt{\left(d_\mathbb{G} \left(u\right)-\frac{2m}{n}\right)^2+\left(d_\mathbb{G} \left(v\right)-\frac{2m}{n} \right)^2}\hfill  \ \ ,
   \end{eqnarray*}
where $\frac{2m}{n}$ denotes the average degree of graph $\mathbb{G} $ and $m$, $n$ represent the set of edges and the set of vertices, respectively. 

It is obvious that we can obtain the general equation for the Sombor indices\cite{GI}, defined as
\begin{eqnarray*}
    SO_a\left(\mathbb{G} \right)=\sum_{uv\in E\left(\mathbb{G} \right)}\sqrt{\left(d_\mathbb{G} \left(u\right)-a \right)^2+\left(d_\mathbb{G} \left(v\right)-a\right)^2}\hfill  \ \ .
\end{eqnarray*}

\begin{figure}[htbp]
    \centering\includegraphics[width=8cm,height=2.5cm]{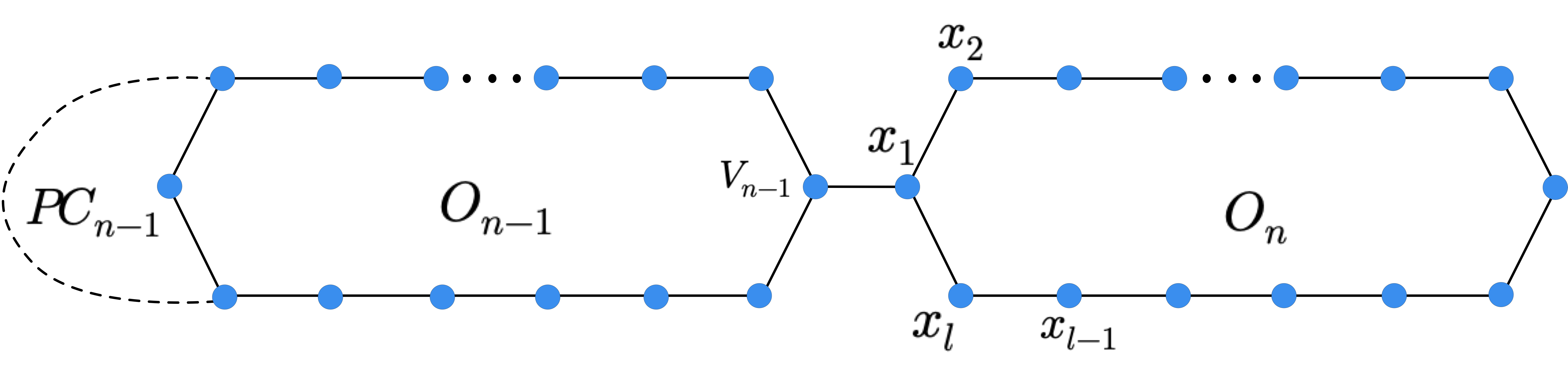}
    \caption{A $l$-polygonal chain $PC_n$ with $n$ polygons, where $O_n$ is the $n$-th polygon.}
    \end{figure}

 \begin{figure}[htbp]
    \centering\includegraphics[width=13cm,height=11cm]{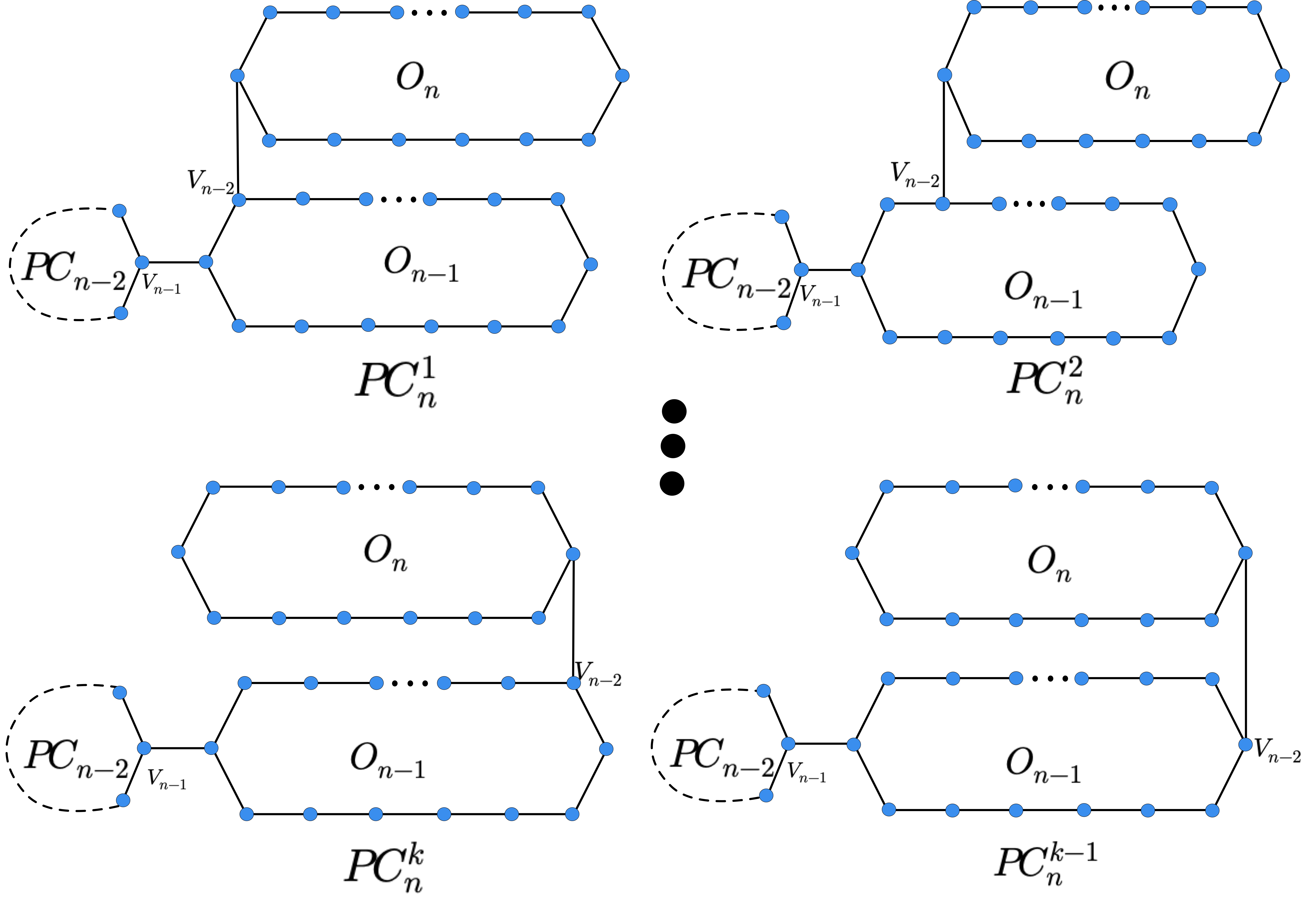}
    \caption{$k$ sorts of permutations in an $l$-polygonal chain.}
     \end{figure}

An $l$-cycles (polygon) $PC_n$ can be seen as a new structure. Between a new polygon with $l$-cycles and a chain of polygons with $n-1$ $l$-cycles, with line segments acting as bridges connecting the endpoints of the segments. When $1\leqslant i \leqslant n$, $O_i$ is known as the $i$-th polygon of the $PC_n$; The polygonal chain is unique when $n = 1,\ 2 $ (see Fig $1$); When $n \geqslant  3$, obtain $PC^i_n$ by joining the vertex of $O_n$ with the vertex $V_{n-1}$ of $O_{n-1}$, where $1\leqslant i\leqslant n$.

For an $l$-cycle random polygon chain $PC_n(n;\ p_1,\ \ldots,\ p_k)$. The transition from $PC_{n-1}$ to $PC_n$ is assumed to be a stable stochastic process, indicated by $p_i$, $p_i$ is a constant, independently of $n$, where $1\leqslant i\leqslant k$ (see Fig 2). In other words, the stated process is viewed as a zeroth-order Markovprocess.

Topological indices and graph invariants based on the distances between graph vertices are widely used to characterize molecular graphs, establish relationships between the structure and properties of molecules, predict the biological activity of compounds, and enable their chemical applications. The Sombor index is one of the more novel topological indices available and is highly correlated with many physical and chemical indices of molecular compounds.
The Sombor index, a new category of degree-based topological molecular descriptors, has received a lot of attention derived from its excellent chemistry applicability. Roberto et al\cite{CRI}, were concerned the results of the Sombor index of chemical graphs.
In\cite{LH.C}, the Sombor index of the chemistry tree was given, and the boiling temperatures of benzenoid hydrocarbons and the (reduced) Sombor index were confirmed to be substantially linked. Existing bounds and extremal results related to the Sombor index and its variants were collected by Liu and You\cite{8}.
In\cite{RJJ},  the authors systematically described the general properties of the Sombor indices.
For additional outstanding Sombor indices results, we recommend that the reader refer to\cite{LZZ,KVRI,LHYL,GNS,MIM}. Meanwhile, more scientific studies were focused on the exploration of the mathematical expected values or variances of some degree-based topological indices.
In\cite{WSKX}, a simple mathematical formula for the expected values of the Wiener index were established in a random cyclooctane chain.
In random hydrocarbon chains, Raza et al\cite{Raza} have discovered the expected values for the several chemistry indices.
The expected values of the basic Randic index for conjugated hydrocarbons were reported by the\cite{LJJ}.
Zhang and Li\cite{ZLL} discussed the expected values of the four types of degree-based topological indices in a random polyphenylene chain.
In\cite{HGK}, the expected values of the Kirchhoff indices in the random polyphenyl and spiro chains were obtained.
Many studies have revealed the expected values and variances of degree-based topological indices for a wide range of substances across this time period, and the reader can refer to\cite{CHL,ZWY,ZJP,TGX}.

As a result of the preceding research, we offered precise analytical formulas for the Sombor indices's expected values and variances in this study.

The structure of the paper is detailed following. In Section $2$, we define a general random polygonal chain and introduce some basic probability theory notions. The distributions of the Sombor indices in a general random chain are established in Section $3$. Based on these distributions, we obtain expressions for their expected values and variances in Section $4$. As an application of Sections $3$ and $4$, in Section $5$, the Sombor indices for the polyonino, pentachain, polyphenyl, and cyclooctane chains, as well as their expected values and variances of the Sombor indices are calculated. The asymptotic behavior of the distributions of the Sombor indices for all the polygonal chains given in this paper is illustrated in Section $6$.

\section{Preliminaries}
 \ \ \ \ For a random polygonal chain $PC_n$, the $SO_a(PC_n)$ are random variables. First, we introduce some of the probability theories that will used in next section.
 \begin{defi}
    The random chain $G_n= G(n;\ p_1,\ p_2,\ \ldots,\ p_k)$ with $n$ identical graphs, a graph $\mathbb{G} $ is given, created by the following ways:
    \begin{flushleft}
    $(1)$ $ \mathbb{G} _2$ is made up for two polygons, see Fig 1.   \\ 
    $(2)$ For each $n>2,\ \mathbb{G} _n$ is created by attaching one $O_n$ to $\mathbb{G} _{n-1}$ in certain ways, resulting in $\mathbb{G} ^1_n, \ \mathbb{G} ^2_n,\ \ldots,\ \mathbb{G} ^k_n$ with probability $p_1,\ p_2,\ \ldots,\ p_k$, respectively, where $\sum_{i= 1}^{k} p_i=1$.
\end{flushleft} 
\end{defi}

Some of the fundamental concepts of probability theory are illustrated below.
 
The polynomial distribution is a binomial distribution extension with n independent replicate experiments, each with k possible results. It is usually denoted by $M (n,\mathbf{p})$, where $n$ represents the number of experiments and the vector  $\mathbf{p}$ is the probability of occurrence of the event. 
\begin{equation*}
    \{\mathbf{p} | \mathbf{p} =(p_1,\ p_2,\ \ldots,\ p_k)^T\in \mathbb{R}^t, i=1,\ 2,\ \ldots,\ k, \ p_i\geqslant 0\ and\  \sum_{ i = 1}^{k} p_i=1\} .
\end{equation*}
 
The sample space of the multinomial distribution is 
 \begin{equation*}
   \mathbb{S}  = \{ \mathbf{X}   | \mathbf{X}  =(x_1,\ x_2,\ \ldots,\ x_k)^T\in \mathbb{Z}^t,\ i=1,\ 2,\ldots ,\ k,\ x_i\geqslant 0\ and  \ \sum_{ i = 1}^{k} x_i=n\} .
\end{equation*}
 
The probability function of the multinomial distribution is
\begin{equation*}
    f(x)=\frac{n!}{\prod_{i=1}^k x_i!} \prod_{i=1}^{k}p_i^{x_i},\ x\in \mathbb{S}  ,
\end{equation*}

Assume that the random vector $\mathbf{X}$ follows the polynomial distribution, which is denoted by $\mathbf{X} \thicksim M(n,\mathbf{p})$. The following conclusion about $X$ can be obtained.
\begin{equation}
    \mathbb{E} (\mathbf{X} )=n\mathbf{p} ,\  \mathbb{V}\mathbf{a} \mathbf{r} (\mathbf{X} )=n(\mathbf{d}\mathbf{i}\mathbf{a}\mathbf{g}(\mathbf{p})-\mathbf{p} \mathbf{p} ^T) .
\end{equation}
 
The bernoulli distribution $B(p_1)$, binomial distribution$M(n,\ p_1)$, and category distribution $C(p)$ are obtained by setting $n=1$ and $k=1$,\ $n>1$ and $k=2$,\ and $n=1$ and $k>2$ in that order. 

The multinomial distribution's two essential properties are described below.

\begin{prop}
    (Addition Rule\cite{PWL}). Let $\mathbf{X} _i\thicksim M(n_i,\mathbf{p})$, where $i=1,\ 2, \ldots,\ k$, for each $X_i$ is an independent vector of each other. Then
    \begin{equation*}
        \sum_{i = 1}^{k}\mathbf{X} _i \thicksim M(\sum_{i = 1}^k n_i,\ \mathbf{p} ) .
    \end{equation*}
\end{prop}

\begin{prop}
    (Marginal Distribution\cite{PWL}). Let $\mathbf{X} =(\mathbf{X} _1,\ \mathbf{X} _2,\ \ldots ,\ \mathbf{X} _k)\thicksim M(n,\ \mathbf{p} )$. Then $X_i\thicksim B(n,\ p_i)$  for each $i=1,\ 2,\ \ldots,\ k$.
\end{prop}

For further information on general probability distributions, please refer to\cite{KPM}. The Lemma 2.3 below is commonly applied.
\begin{lem}
    Let $X$ be a random variable and $a,\ b\in \mathbb{R} $. Then 
    \begin{equation*}
        \mathbb{E} (\mathbf{A} \mathbf{X} +\mathbf{b} )=\mathbf{A}\mathbb{E} (\mathbf{X} )+\mathbb{E} (\mathbf{b} ) ,\ \mathbb{V} \mathbf{a} \mathbf{r}  (\mathbf{A} \mathbf{X} +\mathbf{b} )=\mathbf{A} \mathbb{V}  \mathbf{a} \mathbf{r} (\mathbf{X} )\mathbf{A} ^T.
    \end{equation*}

\end{lem}

A general method of calculating the expected values and variances for Sombor indices in polygonal chain are given.
\begin{thm}
    Let $PC_n$ be a random polygonal chain  of length $n$ and $\mathbf{X} \thicksim B(n-2,\ p_1)$, where $n>2$. Then
     
    \begin{eqnarray*}
        SO(PC_n)&=&A\mathbf{X} +B(n-2)+C ,\\
        \mathbb{E} (SO_a(PC_n))&=&(p_1A+B)(n-2)+C ,\\
        \mathbb{V} \mathbf{a} \mathbf{r} (SO_a(PC_n))&=&A^2(n-2)p_1(1-p_1),
        \end{eqnarray*}
$where$
    \begin{equation*}
        (A,B,C)=\Bigl(A_1-A_2,\ A_2,\ -2A_2+SO_a(PC_n)\Bigr).
    \end{equation*}
\end{thm}

\begin{proof}
    In order to quantify the random variable $SO_a(PC_n)$, we defined a family of $3$-dimensional random vectors $ {\mathbb{Z} _k}$ as follows:
    \begin{eqnarray*}
        \mathbb{Z} _k=
        \begin{cases}
        (1,0,0),       & PC_n=PC^1_n ;\\
        (0,1,0),       &  PC_n=PC^2_n ; \\
        (0,0,1),       &  PC_n=PC^3_n .\\
        \end{cases}
        \end{eqnarray*}

By the definition of the random polygonal chain, we can check that $\mathbb{Z} _k$ follows the categprical distribution $C(1,\ p_1,\ p_2,\ p_3)$ and $\mathbb{Z} _3,\ \mathbb{Z} _4,\ \ldots ,\ \mathbb{Z} _n$ are independent.

For each $k=3,\ 4,\ \ldots,\ n$, $SO_a(PC_k)$ can be quantified as
\begin{equation}
    SO_a(PC_k)=\Bigl(SO_a(PC^1_k),\ SO_a(PC^2_k),\ SO_a(PC^3_k)\Bigr)\mathbb{Z}_k .
\end{equation}
 
By the definition of Sombor indices, for each $k=3,\ 4,\ \ldots,\ n$, and $i=1,\ 2,\ 3$, we have 
\begin{equation*}
    SO_a(PC^i_k)-SO_a(PC_{k-1})=SO_a(PC^i_3)-SO_a(PC_2) ,
\end{equation*}
we denoted $A_i=SO_a(PC^i_3)-SO_a(PC_2),\ i=1,\ 2,\ 3$. Then 
\begin{equation}
    SO_a(PC^i_k)=SO_a(PC_{k-1})+A_i,\ i=1,\ 2,\ 3.
\end{equation}

Associated (2.2) with (2.3), the $SO_a(PC_k)$ satisfies the following recursive relation
\begin{equation}
  SO_a(PC_k)=SO_a(PC_{k-1})+(A_1,\ A_2,\ A_3)\mathbb{Z} _k,
\end{equation}
\begin{equation}
    SO_a(PC_n)=(A_1,A_2,A_3)\mathbf{X} +SO_a(PC_2),\ \mathbf{X} =(X_1,\ X_2,\ X_3)^T=\sum_{k = 3}^{n}  \mathbb{Z}  _k,
\end{equation}
where $\mathbf{X} $ follows the Multinomial distribution $M (n-2,\ p_1,\ p_2,\ p_3)$ by Proposition 2.1.

Thus, we obtain
    \begin{align*}
        A_1&=SO_a(PC^1_3)-SO_a(PC_2) ,\\
        A_2&=SO_a(PC^2_3)-SO_a(PC_2)=A_3.
    \end{align*}
    
By (2.5), we have 
\begin{eqnarray*}
    SO_a(PC_n)
    &=&(A_1,A_2,A_2)\mathbf{X} +SO_a(PC_2)\\
    &=&(A_1-A_2)(1,0,0)\mathbf{X} +A_2(1,1,1)\mathbf{X} +TI(PC_2)\\
    &=&(A_1-A_2)X_1+A_2(n-2)+SO_a(PC_2),
    \end{eqnarray*}
where $X_1$ follows the binomial distribution $B(n-2,\ p_1)$ by Proposition 2.2.

Put $A=A_1-A_2,\ B=A_2,\ C=SO_a(PC_2)$.
Applying Lemma 2.3, we obtain
\begin{equation*}
   \mathbb{E} \bigl(SO_a(PC_n)\bigr)=(p_1A+B)(n-2)+C ,\ \mathbb{V} \mathbf{a} \mathbf{r} \bigl(SO_a(PC_n)\bigr) =A^2(n-2)p_1(1-p_1) .
\end{equation*}
\end{proof}

\section{The distribution for $SO_a(\mathbb{G} _n)$}
\ \ \ \ We discover that $SO_a(PC_n)$ are regarded as random variables. Then, the distributions of $SO_a(PC_n)$ of a $(2k+1)$-polygonal chain and $SO_a(PC_n)$ of a $2k$-polygonal chain are presented in this section.
\begin{thm}
   Let $\mathbb{G} _n(n>2)$ be an random polygonal chain also with connect constants $\{ A_i\} ^k_{i=1}$ for Sombor indices, by Definition $2.1$.
    Then 
    \begin{equation*}
        SO_a(\mathbb{G} _n)=\mathbf{A} ^T X+SO_a(\mathbb{G} _2) ,
    \end{equation*}
    $where$ $\mathbf{A} ^T=(A_1,\ A_2,\ \ldots,\ A_k)$ $and$ $\mathbf{X}$ follow the $M(n-2,\ \mathbf{p} ),\ \mathbf{p} =(p_1,\ p_2,\ \ldots,\ p_k)$.
    \begin{proof}
        Associated $(2.2),\ (2.3),\ (2,4)$ with $(2.5)$ can be obtained.
    \end{proof}
\end{thm}
\begin{cor}
    Let $PC_n$ be a $(2k+1)$-polygonal chain of length $n$ and $X\sim B(n-2,\ p_1)$, where $k\geqslant 2$ and $n>2$. Then $SO_a(PC_n)=AX+Bn+C$, where
    \begin{eqnarray*}
        A&=&2\sqrt{2} \left\lvert 2-a\right\rvert+\sqrt{2}\left\lvert 3-a\right\rvert  +2v_1, \\
        B&=&\sqrt{2}(2k-3)\left\lvert 2-a\right\rvert +\sqrt{2} \left\lvert 3-a\right\rvert +v_1 , \\
        C&=&4\sqrt{2} \left\lvert 2-a\right\rvert -\sqrt{2} \left\lvert 3-a\right\rvert +8v_1 .
        \end{eqnarray*}

Specifically, we have 
    \begin{eqnarray*}
        SO(PC_n)&=&(6\sqrt{2} +2\sqrt{13} )X+(4\sqrt{2}k-3\sqrt{2}+\sqrt{13}  )n+5\sqrt{2}+8\sqrt{13}  ,\\
            SO_{red}(PC_n)&=&(4\sqrt{2} +2\sqrt{5} )X+(2\sqrt{2}k -\sqrt{2}+\sqrt{5}  )n+3\sqrt{2}+8\sqrt{5} ,\\
            SO_{avr}(PC_n)&=&A_1X+B_1n+C_1,
        \end{eqnarray*}
where 
    \begin{eqnarray*}
         A_1&=&\frac{\sqrt{2}(2k+3) +2\mu _1}{n(2k+1)}  ,\\ 
         B_1&=&\frac{6\sqrt{2} nk-7\sqrt{2}n+4\sqrt{2} k-8\sqrt{2}  +\mu _1}{n(2k+1)} ,\\
         C_1&=&\frac{-4\sqrt{2}nk+9\sqrt{2} n+11\sqrt{2} +8\mu _1 }{n(2k+1)} ,\\
         v_1&=&\sqrt{2a^2-10a+13} ,\\
         \mu _1&=&\sqrt{4n^2k^2-4n^2k-8nk+5n^2+12n+8} .
        \end{eqnarray*}
\end{cor}

\begin{proof}
    According to the concept of Sombor indices, we can obtain the Sombor indices of a random $(2k+1)$-polygonal chain. Specific proofs are as follows.
    \begin{eqnarray*}
        SO_a(PC_2)&=&(4k-2)\sqrt{2} \left\lvert 2-a\right\rvert +\sqrt{2} \left\lvert 3-a\right\rvert +4v_1,\\
             SO_a(PC^1_3)&=&(6k-4)\sqrt{2}\left\lvert 2-a\right\rvert  +3\sqrt{2} \left\lvert 3-a\right\rvert +6v_1 ,\\
             SO_a(PC^2_3)&=&(6k-5)\sqrt{2} \left\lvert 2-a\right\rvert +2\sqrt{2} \left\lvert 3-a\right\rvert+8v_1 ,\\
             &\vdots & \\
             SO_a(PC^k_3)&=&(6k-5)\sqrt{2} \left\lvert 2-a\right\rvert +2\sqrt{2} \left\lvert 3-a\right\rvert+8v_1,
           \end{eqnarray*}
then relevant variables of $PC_n$ are derived by 
    \begin{eqnarray*}
        A_1&=&SO_a(PC^1_3)-SO_a(PC_2)\\
           &=&(2k-2)\sqrt{2} \left\lvert 2-a\right\rvert +2\sqrt{2} \left\lvert 3-a\right\rvert +2v_1,\\
            A_2&=&SO_a(PC^2_3)-SO_a(PC_2)\\ 
               &=&(2k-3)\sqrt{2} \left\lvert 2-a\right\rvert +\sqrt{2} \left\lvert 3-a\right\rvert +4v_1 ,
          \end{eqnarray*}
where $v_1=\sqrt{2a^2-10a+13}$.

Applying (2.3), one obtains
    \begin{equation*}
        SO_a(PC_n)=(A_1-A_2)X+A_2n-2A_2+SO_a(PC_2),\ X\sim B(n-2,\ p_1).
    \end{equation*}
    
First part can be obtained by right away. If $\left\lvert V(PC_n)\right\rvert =n(2k+1)$ and $\left\lvert E(PC_n)\right\rvert =2n(k+1)-1$, then the average degree of $PC_n$ is $\frac{4n(k+1)+2}{n(2k+1)}$. The proof is completed by putting $a=0,\ 1,\ \frac{4n(k+1)+2}{n(2k+1)}$ into the equation, respectively.

\end{proof}

\begin{cor}
    Let $PC_n$ be a random $2k$-polygonal chain of length $n$ and $X\sim B(n-2,\ p_1)$, where $k\geqslant 2$ and $n>2$. Then $SO_a(PC_n)=AX+Bn+C$, where
    \begin{eqnarray*}
        A&=&2\sqrt{2} \left\lvert 2-a\right\rvert-2v_1 , \\
        B&=&\sqrt{2} ((2k-3)\left\lvert 2-a\right\rvert +\left\lvert 3-a\right\rvert )+4v_1 , \\
        C&=&\sqrt{2} \left\lvert 2-a\right\rvert -\sqrt{2} \left\lvert 3-a\right\rvert -4v_1 .
        \end{eqnarray*}

In particular, we have 
    \begin{eqnarray*}
        SO(PC_n)&=&(2\sqrt{2} -2\sqrt{13} )X+(4\sqrt{2}k+4\sqrt{13}-3\sqrt{2}  )n-\sqrt{2}-4\sqrt{13}  ,\\
            SO_{red}(PC_n)&=&(\sqrt{2} -2\sqrt{5} )X+(2k-1)\sqrt{2}n+4\sqrt{5}n  -\sqrt{2} k-4\sqrt{5}  ,\\
            SO_{avr}(PC_n)&=&A_2X+B_2n+C_2,
          \end{eqnarray*}
where 
    \begin{eqnarray*}
        A_2&=&\frac{2\sqrt{2}n-2\mu _2-2\sqrt{2}  }{nk} , \\
        B_2&=&\frac{3\sqrt{2}nk-2\sqrt{2}k-4\sqrt{2} n+4\sqrt{2} +4\mu _2 }{nk} ,\\
        C_2&=&\frac{-\sqrt{2} nk+2\sqrt{2}n-2\sqrt{2} -4\mu _1 }{nk} ,\\
        v_1&=&\sqrt{2a^2-10a+13},\\
        \mu _2&=&\sqrt{n^2k^2-2n^2k+2nk+2n^2-4n+2} .
    \end{eqnarray*}
\end{cor}

\begin{proof}
    According to the definition of Sombor indices, we can obtain the Sombor indices of a random $2k$-polygonal chain.
    \begin{eqnarray*}
        SO_a(PC_2)&=&\sqrt{2} \Bigl((4k-5\left\lvert 2-a\right\rvert )+\left\lvert 3-a\right\rvert \Bigr) +4v_1 ,\\
             SO_a(PC^1_3)&=&\sqrt{2} \Bigl((6k-7\left\lvert 2-a\right\rvert )+3\left\lvert 3-a\right\rvert \Bigr)+6v_1  ,\\
             SO_a(PC^2_3)&=&\sqrt{2} \Bigl((6k-8\left\lvert 2-a\right\rvert )+2\left\lvert 3-a\right\rvert \Bigr) +8v_1  ,\\
             &\vdots & \\
             SO_a(PC^k_3)&=&\sqrt{2} \Bigl((6k-8\left\lvert 2-a\right\rvert )+2\left\lvert 3-a\right\rvert \Bigr) +8v_1  ,
          \end{eqnarray*}
then relevant variables of $PC_n$ are given by
    \begin{eqnarray*}
        A_1&=&SO_a(PC^1_3)-SO_a(PC_2)\\
        &=&\sqrt{2} \Bigl((2k-2\left\lvert 2-a\right\rvert )+\left\lvert 3-a\right\rvert \Bigr)+2v_1  ,\\
            A_2&=&SO_a(PC^2_3)-SO_a(PC_2)\\
            &=&\sqrt{2} \Bigl((2k-3\left\lvert 2-a\right\rvert )+\left\lvert 3-a\right\rvert \Bigr)+4v_1  .
          \end{eqnarray*}

Appiying (2.3), we obtain
    \begin{equation*}
        SO_a(PC_n)=(A_1-A_2)X+A_2n-2A_2+SO_a(PC_2),\ X\sim B(n-2,\ p_1).
    \end{equation*}

Setting $A=A_1-A_2,\ B=A_2$ and $C=-2A_2+SO_a(PC_2)$, we promptly get the first component. Setting $A = A1-A2, B = A2$, and $C =-2A_2+ SO_a(PC_2)$, we promptly get the first component. Since $\left\lvert V(PC_n)\right\rvert =2nk$ and $\left\lvert E(PC_n)\right\rvert =2nk+n-1$, then the $\frac{2m}{n}$ of $PC_n$ is $\frac{2nk+n-1}{nk} $. The proof is performed directly by setting $a=0,\ 1,\ \frac{2nk+n-1}{nk}$.
\end{proof}

\section{The expected values and variances for $SO_a(\mathbb{G} _n)$}
\ \ \ \ The exact analytical equations of $\mathbb{E} (SO_a(PC_n))$ and $\mathbb{V} \mathbf{a} \mathbf{r} (SO_a(PC_n))$ of a $(2k+1)$-polygonal chain and a $2k$-polygonal chain are discussed in this section, respectively.

\begin{thm}
    Let $\mathbb{G} _n$ be a random polygonal chain. The expected values and variances of $SO_n(\mathbb{G} _n)$ are calculated by 
    \begin{eqnarray*}
        \mathbb{E} (\mathbb{G} _n)&=&\left(\sum_{i = 1}^{k} A_i p_i \right)(n-2)+SO_a(G_2) ,\\
            \mathbb{V} \mathbf{a} \mathbf{r} (\mathbb{G} _n) &=& \left(\sum_{i = 1}^{k} A^2_i p_i-(\sum_{i = 1}^{k}A_i p_i)^2 \right) .
          \end{eqnarray*}
          \begin{proof}
              $SO_a(\mathbb{G} _n)\mathbf{A} ^T+SO_a(\mathbb{G} _2),\ \mathbf{X} \sim M(n-2,\mathbf{p} ),\ \mathbf{p} =(p_1,\ p_2,\ \ldots,\ p_k)^T. $
              
By Lemma 2.3, we obtain 
              \begin{eqnarray*}
                \mathbb{E} \Bigl(SO_a(\mathbb{G} _n)\Bigl)&=&E(\mathbf{A}^T\mathbf{X}+SO_a(\mathbb{G} _2)) \\
                &=&\mathbf{A}^T E(\mathbf{X} )+SO_a(\mathbb{G} _2) \\
                &=&(\sum_{i = 1}^{k} A_i p_i )(n-2)+SO_a(\mathbb{G} _2),\\
                \mathbb{V} \mathbf{a} \mathbf{r} \Bigl(SO_a(\mathbb{G} _n)\Bigl)&=&Var(\mathbf{A}^T\mathbf{X}+SO_a(\mathbb{G} _2))\\
                &=&\mathbf{A}^T Var(\mathbf{X})\mathbf{A}\\
                &=&\Bigl(\sum_{i = 1}^{k} A_i p_i-(\sum_{i=1}^{k}A_i p_i)^2\Bigr) (n-2).
                  \end{eqnarray*}
          \end{proof}
\end{thm}

\begin{cor}
    Let $PC_n$ be a random $(2k+1)$-polygonal chain of length $n$, where $k\geqslant 2$ and $n>2$. Then 
    \begin{equation*}
        \mathbb{E} \Bigl(SO_a(PC_n)\Bigl)=Mn+N,\ \mathbb{V} \mathbf{a} \mathbf{r} \Bigl(SO_a(PC_n)\Bigl)=Pn+Q,
    \end{equation*}
where
    \begin{eqnarray*}
       M&=&\sqrt{2}\left\lvert 2-a\right\rvert(p_1+2k-3)+\sqrt{2}(1+p_1)\left\lvert 3-a\right\rvert  +(4-2p_1)v_1 , \\
       N&=&(-2\sqrt{2}p_1-4\sqrt{2} k+10\sqrt{2}) \left\lvert 2-a\right\rvert -3\sqrt{2}\left\lvert 3-a\right\rvert +4p_1v_1  , \\
       A^2&=&3v_1^2+2\left\lvert 2-a\right\rvert \left\lvert 3-a\right\rvert -2\sqrt{2} (\left\lvert 2-a\right\rvert +\left\lvert 3-a\right\rvert ),\\
       P&=&A^2p_1(1-p_1)  ,\\
       Q&=&-2A^2p_1(1-p_1),\\
       v_1&=&\sqrt{2a^2-10a+13} .
          \end{eqnarray*}

In particular, we have 
    \begin{eqnarray}
        \begin{aligned}
        \mathbb{E} \Bigl(SO(PC_n)\Bigr)
        &=\Bigl((5\sqrt{2}-2\sqrt{13} )p_1+4\sqrt{2}k-3\sqrt{2}+4\sqrt{13}\Bigr)n \\
        &+(4\sqrt{13}-4\sqrt{2})p_1-8\sqrt{2}k+11\sqrt{2}   , 
        \end{aligned}
    \end{eqnarray} 
    \begin{eqnarray}
        \begin{aligned}
       \mathbb{E} \Bigl(SO_{red}(PC_n) \Bigr)
        &=\Bigl((3\sqrt{2}-2\sqrt{5})p_1+2\sqrt{2}k+4\sqrt{5} -\sqrt{2}\Bigr)n  \\
        &+(4\sqrt{5} -2\sqrt{2} ) p_1-4\sqrt{2}k+4\sqrt{2}  , 
    \end{aligned}
    \end{eqnarray} 
    \begin{eqnarray}
        \mathbb{E} \Bigl(SO_{avr}(PC_n)\Bigl)=M_1n+N_1,
     \end{eqnarray}
 where 
     \begin{eqnarray*}
        M_1&=&\frac{\sqrt{2} np_1(2k+1)+\sqrt{2}n(6k-7)+4\sqrt{2}(k-2)+(4-2p_1)\mu _1}{n(2k+1)} ,\\
        N_1&=&\frac{-4\sqrt{2}np_1-14\sqrt{2} nk+(4\mu _1-4\sqrt{2} )p_1-8\sqrt{2} +23\sqrt{2} n+26\sqrt{2} }{n(2k+1)} ,\\
        v_1&=&\sqrt{2a^2-10a+13} ,\\
        \mu _1&=&\sqrt{4n^2k^2-4n^2k-8nk+5n^2+12n+8} .
           \end{eqnarray*}

    \begin{eqnarray}
        \mathbb{V} \mathbf{a} \mathbf{r} \Bigl(SO(PC_n)\Bigl) 
        &=&(16\sqrt{26}+84 )p_1(1-p_1)(n-2),
    \end{eqnarray} 
    \begin{eqnarray}
        \mathbb{V} \mathbf{a} \mathbf{r}\bigl(SO_{red}(PC_n)\bigr)
        &=&(8\sqrt{10}+28 )(n-2)p_1(1-p_1), 
        \end{eqnarray} 
\begin{eqnarray}
    \mathbb{V} \mathbf{a} \mathbf{r}\bigl(SO_{avr}(PC_n)\bigr) =\widetilde{\sigma } (n-2)p_1(1-p_1),   
    \end{eqnarray} 
 where
        \begin{eqnarray*}
            \widetilde{\sigma }
            =\frac{16(\sqrt{2} +1)\sqrt{4n^2k^2-4n^2k-8nk+5n^2+12n+8} }{n^2(2k+1)^2} 
            -\frac{64(4n^2k^2-4n^2k+2n+1)}{n^2(2k+1)^2} +84.
         \end{eqnarray*}

\end{cor}
\begin{cor}
    Let $PC_n$ be a random $(2k)$-polygonal chain of length $n$, where $k\geqslant 2$ and $n>2$. Then
    \begin{equation*}
        \mathbb{E} \Bigl(SO_a(PC_n)\Bigl)=Mn+N,\ \mathbb{V} \mathbf{a} \mathbf{r}\Bigl(SO_a(PC_n)\Bigl)=Pn+Q,
    \end{equation*}
 where
    \begin{eqnarray*}
       M&=&\sqrt{2}\Bigl((p_1+2k-3)\left\lvert 2-a\right\rvert+\left\lvert 3-a\right\rvert \Bigr) +(4-2p_1)v_1 , \\
       N&=&\sqrt{2}\left\lvert 2-a\right\rvert \left(-2p_1-4k+7\right) -3\sqrt{2} \left\lvert 3-a\right\rvert -2\sqrt{2} (4-2p_1)v_1  , \\
       P&=&\Bigl(16a^2-72a+84-8\sqrt{2}v_1\left\lvert 2-a\right\rvert\Bigl)p_1(1-p_1) ,\\
       Q&=&-2P.
          \end{eqnarray*}

 In particular, we have 
    \begin{eqnarray}
        \begin{aligned}
        \mathbb{E} \Bigl(SO(PC_n)\Bigl)
        &=\Bigl(2p_1(\sqrt{2}-\sqrt{13} )+4\sqrt{2}k-4\sqrt{13}-3\sqrt{2}\Bigl)n\\
        &+4p_1(\sqrt{26}-\sqrt{2} ) +8\sqrt{2}k+5\sqrt{2}-8\sqrt{26}    , 
     \end{aligned}
    \end{eqnarray} 
    \begin{eqnarray}
        \begin{aligned}
        \mathbb{E} \Bigl(SO_{red}(PC_n)\Bigl)
        &=\Bigl(\sqrt{2} (p_1+2k-1)+(4-2p_1)\sqrt{5} \Bigl) n\\
        &-2\sqrt{2} p_1+4\sqrt{10} p_1-4\sqrt{2} k-8\sqrt{10} +\sqrt{2}  , 
    \end{aligned}
    \end{eqnarray} 
    \begin{eqnarray}
        \begin{aligned}
            \mathbb{E}  \Bigl(SO_{avr}(PC_n)\Bigl)
            &=\frac{-2\sqrt{2}np_1+2\sqrt{2}p_1+8\sqrt{2}n-6\sqrt{2}}{k} \\
            &+\frac{4np_1-8n-4}{nk} \mu _2 ,
        \end{aligned}    
    \end{eqnarray}
    \begin{equation}
        \mathbb{V} \mathbf{a} \mathbf{r}\Bigl(SO(PC_n)\Bigl) =(84-16\sqrt{26} )p_1(1-p_1)(n-2),
    \end{equation}

    \begin{equation}   
        \mathbb{V} \mathbf{a} \mathbf{r}\Bigl(SO_{red}(PC_n)\Bigl) =(28-16\sqrt{10} )(n-2)p_1(1-p_1),
        \end{equation}
    \begin{eqnarray}
        \begin{aligned}
            \mathbb{V} \mathbf{a} \mathbf{r}\Bigl(SO_{avr}(PC_n)\Bigl)
        &=\frac{-80n^2k^2-8n^2k+8nk+16n^2-32n+16}{n^2k^2} \\
        &+\frac{-8\sqrt{2} (n-1)\mu _2 }{n^2k^2} ,
            \end{aligned}
        \end{eqnarray} 
where 
\begin{eqnarray*}
    v_1&=&\sqrt{2a^2-10a+13}  ,\\
    u_2&=&\sqrt{n^2 k^2-2n^2 k+2nk+2n^2-4n+2} .
\end{eqnarray*}
\end{cor}

\section{Application of Corollaries 4.2 and 4.3}
\ \ \ \ The expected values and variances of topological indices of random polygonal chains have received much attention in scientific research, such as random polygonal chains, random pentagonal chains, random polyphenyl chains, and random cyclooctane chains. This section studied the Sombor indices distribution of these polygonal chains as well as their expected values and variances. According to the common research methods presented in 4.2 and 4.3, there was the following analysis.

A random polyonino chain is a random modified domino chain with $n$ squares\cite{WSX}, represented by $GPC_n(n;\ p_1,\ 1-p_1)$. Due to the numerous intriguing combinatorial subjects that result from them, for example, the dominance problems\cite{CEJ,HSC}, car dominoes, enumeration problems\cite{VLG,LNH} with perfect matching, and so on, generalized domino graphs have aroused the curiosity about certain mathematicians.

In a random generalized poly-energy chain, there are two sorts of local permutations, due to the definition of a random $l$-poly energy chain (see Fig. 3). Based on the calculation of the Sombor indices for an random polygonal chain presented in Section 3, the generalized formula of the Sombor indices of $GPC_n$ is obtained by 
\begin{eqnarray*}
    SO_a(GPC_n)
    &=&(2\sqrt{2}\left\lvert 2-a\right\rvert -2v_1 )X\\
    &+&\Bigl(\sqrt{2} (\left\lvert 2-a\right\rvert+\left\lvert 3-a\right\rvert)+4v_1\Bigr) n\\
    &+&\sqrt{2} \left\lvert 2-a\right\rvert -\sqrt{2} \left\lvert 3-a\right\rvert -4v_1 ,
       \end{eqnarray*}
where $v_1=\sqrt{2a^2-10a+13}$.

\begin{figure}[htbp]
    \centering\includegraphics[width=13cm,height=9cm]{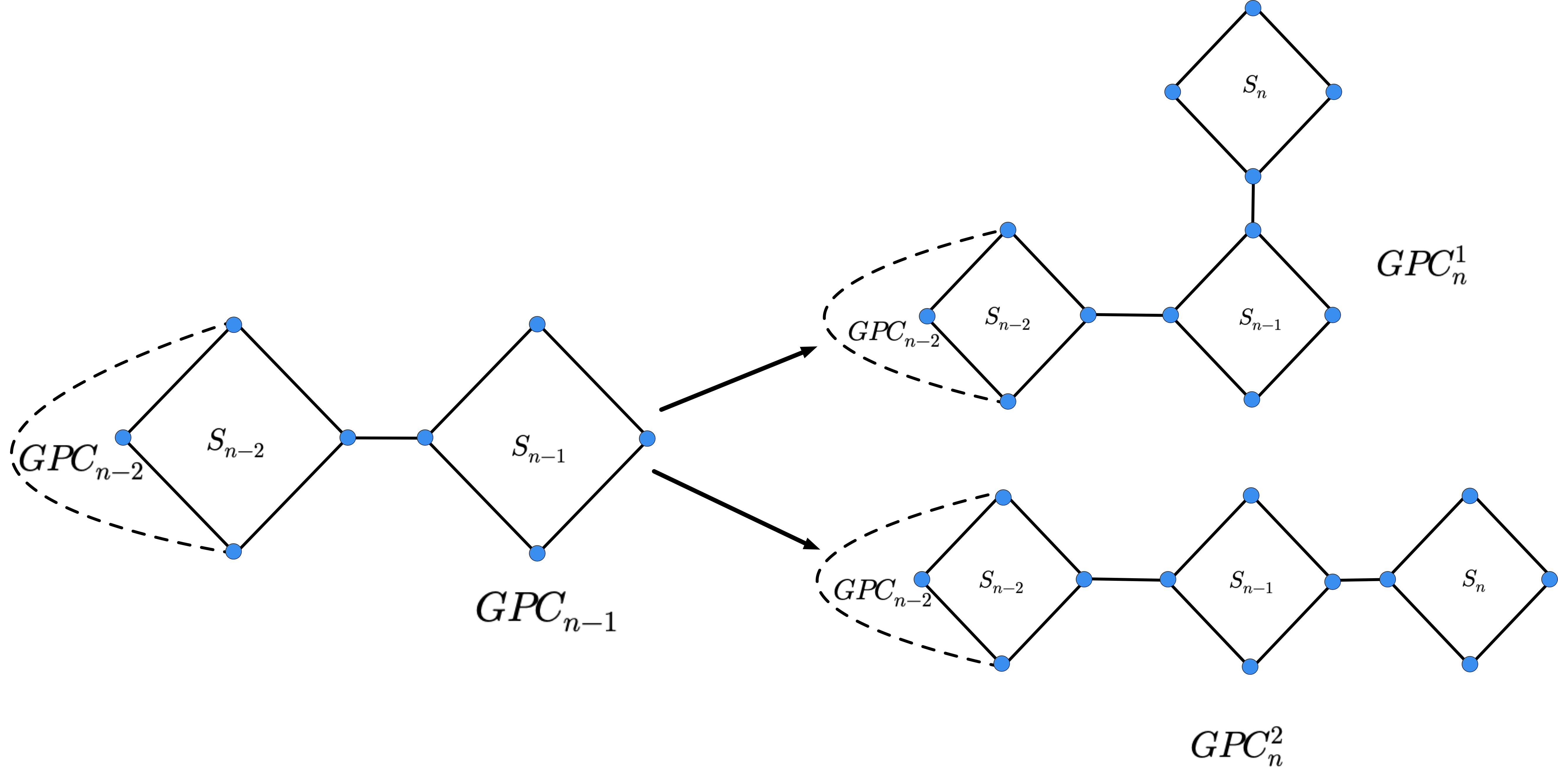}
    \caption{Two sorts of local permutations in a random polyonino chain.}
    \end{figure}
\begin{figure}[htbp]
        \centering\includegraphics[width=13cm,height=10cm]{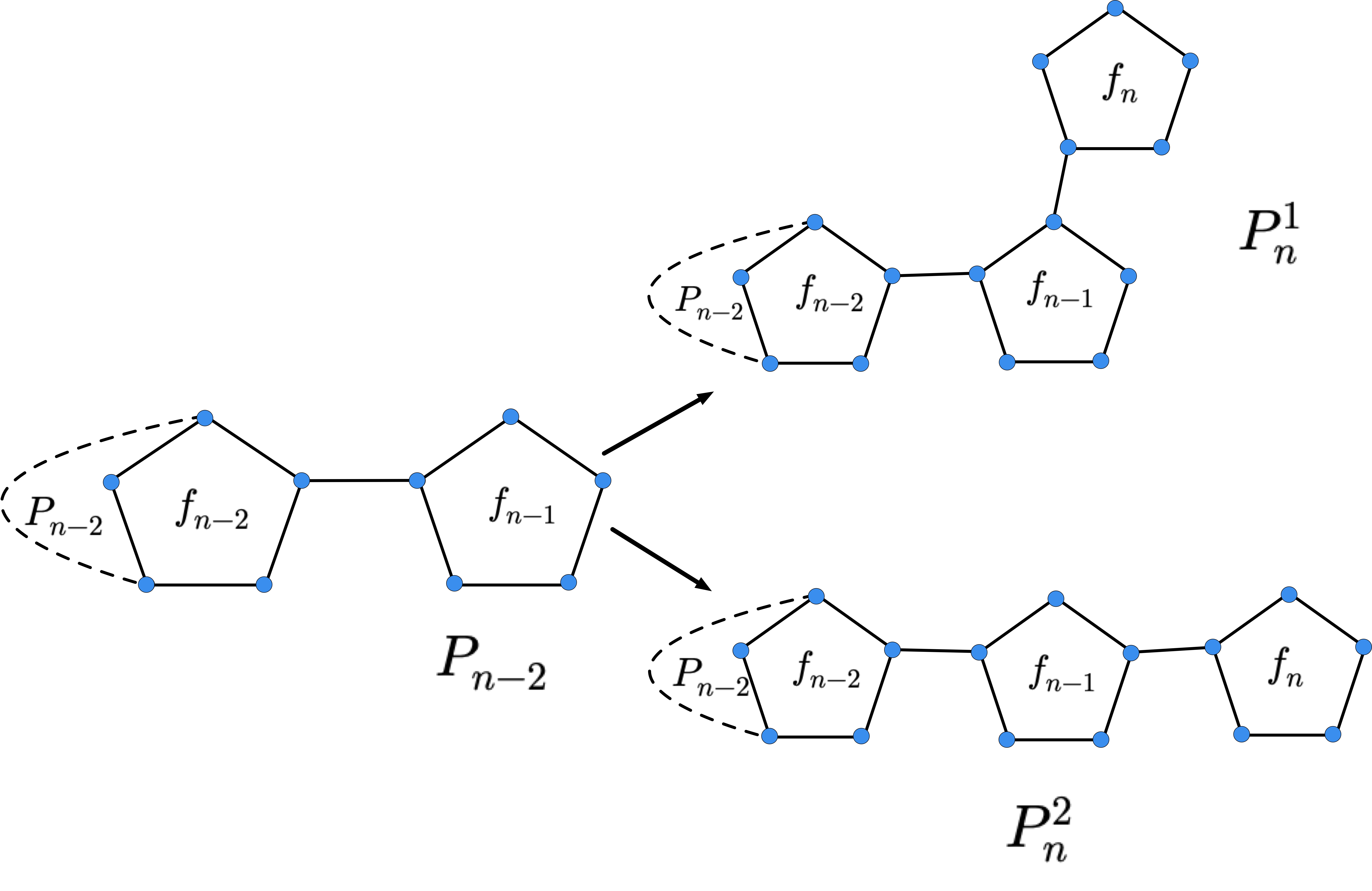}
         \caption{Two sorts of local permutations in a random pentachain.}
    \end{figure}
\begin{thm}
    The expected value and variance of the Sombor indices for $GPC_n$ are given by 
     \begin{eqnarray*}
         \mathbb{E} \Bigl(SO_a(GPC_n)\Bigr)
            &=&\Bigl(\sqrt{2}(1+p_1)\left\lvert 2-a\right\rvert+\sqrt{2}\left\lvert 3-a\right\rvert+(4-2p_1)v_1\Bigr)n\\
            &+&\sqrt{2} (-2p_1-1)\left\lvert 2-a\right\rvert -3\sqrt{2} \left\lvert 3-a\right\rvert -2\sqrt{2} (4-2p_1)v_1 ,
                \end{eqnarray*}
      \begin{eqnarray*}
        \mathbb{V} \mathbf{a} \mathbf{r}\Bigl(SO_a(GPC_n)\Bigr)
        =\Bigl(16a^2 -72a+84-8\sqrt{2}v_1\left\lvert 2-a\right\rvert \Bigr)(n-2)p_1(1-p_1),
        \end{eqnarray*}
        where $v_1=\sqrt{2a^2-10a+13}$.
            \end{thm}

 A random pentachain $P(n;\ p_1,\ 1 - p_1)$ of length $n$ with $n$ pentagons\cite{HOO,RNP}, labeled it by $P_n$.
 
 According to the concept of a random $l$-polygon chain, a random pentachain has two sorts of local permutations (see Fig 4). Depending on the calculation of the Sombor indices for an random polygonal chain presented in Section 3, the generalized formula of the Sombor indices of $P_n$ is obtained by 
 \begin{eqnarray*}
    SO_a(P_n)
    &=&\Bigl(2\sqrt{2}\left\lvert 2-a\right\rvert+\sqrt{2} \left\lvert 3-a\right\rvert  +2v_1\Bigr)X\\
    &+&\Bigl(\sqrt{2} \left\lvert 2-a\right\rvert +\sqrt{2} \left\lvert 3-a\right\rvert +v_1\Bigr) n\\
    &+&4\sqrt{2} \left\lvert 2-a\right\rvert -\sqrt{2} \left\lvert 3-a\right\rvert +8v_1 ,
       \end{eqnarray*}
where $v_1=\sqrt{2a^2-10a+13}$.

 \begin{thm}
    The expected value and variance of the Sombor indices for $P_n$ are given by
\begin{eqnarray*}
    \mathbb{E} \Bigl(SO_a(P_n)\Bigr)
    &=&\Bigl(\sqrt{2} ((1+p_1)\left\lvert 2-a\right\rvert+\left\lvert 3-a\right\rvert  )+(4-2p_1)v_1\Bigr) n\\
    &+&(2\sqrt{2}-2\sqrt{2}p_1 )\left\lvert 2-a\right\rvert -3\sqrt{2} \left\lvert 3-a\right\rvert +4p_1v_1 ,
\end{eqnarray*}     
 \begin{eqnarray*}
    \mathbb{V} \mathbf{a} \mathbf{r}\Bigl((SO_a(P_n)\Bigr)
    =A^2(n-2)p_1(1-p_1),
\end{eqnarray*}
where 
\begin{eqnarray*}
     A^2&=&3v_1^2+2\left\lvert 2-a\right\rvert\left\lvert 3-a\right\rvert -2\sqrt{2} \Bigl(\left\lvert 2-a\right\rvert +\left\lvert 3-a\right\rvert \Bigr) ,\\
            v_1&=&\sqrt{2a^2-10a+13}.
             \end{eqnarray*}

            \end{thm}

Random polyphenyl chain of length $n$, with $n$ hexagons, labeled by $PPC_n(n;\ p_1,\ p_2,\ 1-p_1-p_2)$. 
Polygonal chains are nothing more than an unbranched heavy hydrocarbons graph with modifications that are widely seen in chemical synthesis, medicinal synthesis, and heat exchangers, and they have long aroused the curiosity of chemists\cite{YWF,HGM}.
There are three types of configurations according to the random polygonal chain concept, as shown in Fig 5. Based on the calculation of the Sombor indices for an random polygonal chain presented in Section 3, the generalized formula of the Sombor indices of $PPC_n$ is obtained by
\begin{eqnarray*}
    SO_a(PPC_n)
    &=&\Bigl(2\sqrt{2}\left\lvert 2-a\right\rvert-2v_1\Bigr) X\\
    &+&\Bigl(3\sqrt{2} \left\lvert 2-a\right\rvert +\sqrt{2} \left\lvert 3-a\right\rvert +4v_1\Bigr) n \\
    &+&\sqrt{2} \left\lvert 2-a\right\rvert -\sqrt{2} \left\lvert 3-a\right\rvert -4v_1 ,
       \end{eqnarray*}
       where $v_1=\sqrt{2a^2-10a+13}$.

\begin{thm}
    The expected value and variance of the Sombor indices for $PPC_n$ are derived by 
\begin{eqnarray*}
    \mathbb{E} \Bigl(SO_a(PPC_n)\Bigr)
    &=&\Bigl(\sqrt{2} ((3+p_1)\left\lvert 2-a\right\rvert+\sqrt{2} \left\lvert 3-a\right\rvert  )+(4-2p_1)v_1\Bigr) n\\
    &-&5\left\lvert 2-a\right\rvert -3\sqrt{2} \left\lvert 3-a\right\rvert -2\sqrt{2} (4-2p_1)v_1 ,
\end{eqnarray*}
\begin{eqnarray*}
    \mathbb{V} \mathbf{a} \mathbf{r}\Bigl(SO_a(PPC_n)\Bigr)
    =\Bigl(16a^2-72a+84-8\sqrt{2} v_1\left\lvert 2-a\right\rvert \Bigr)(n-2)p_1(1-p_1).
    \end{eqnarray*}
            \end{thm}

\begin{figure}[htbp]
        \centering\includegraphics[width=13cm,height=6cm]{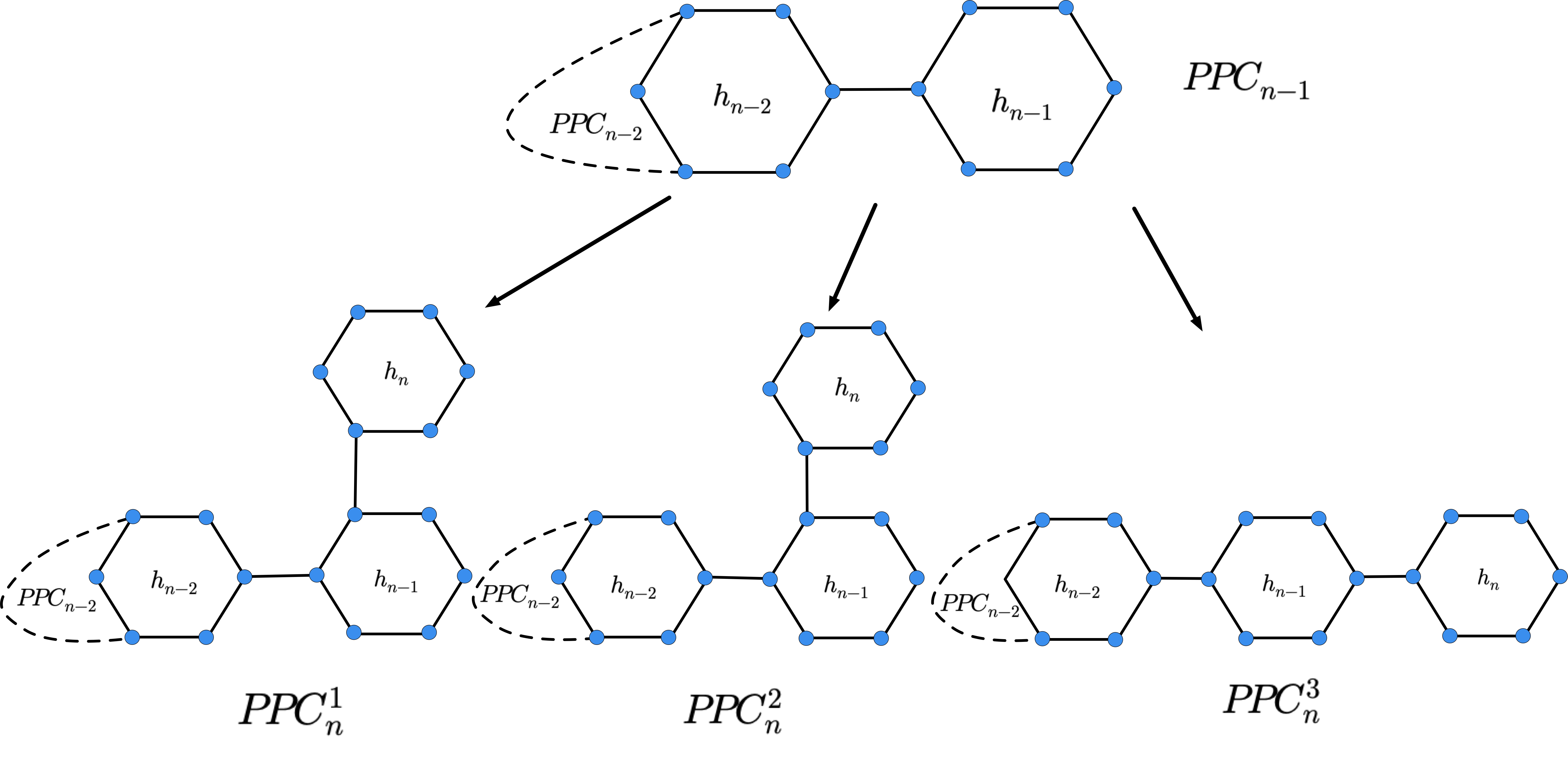}
        \caption{Three sorts of permutations in a random polyphenyl chain.}
        \end{figure}
\begin{figure}[htbp]
    \centering\includegraphics[width=13cm,height=10cm]{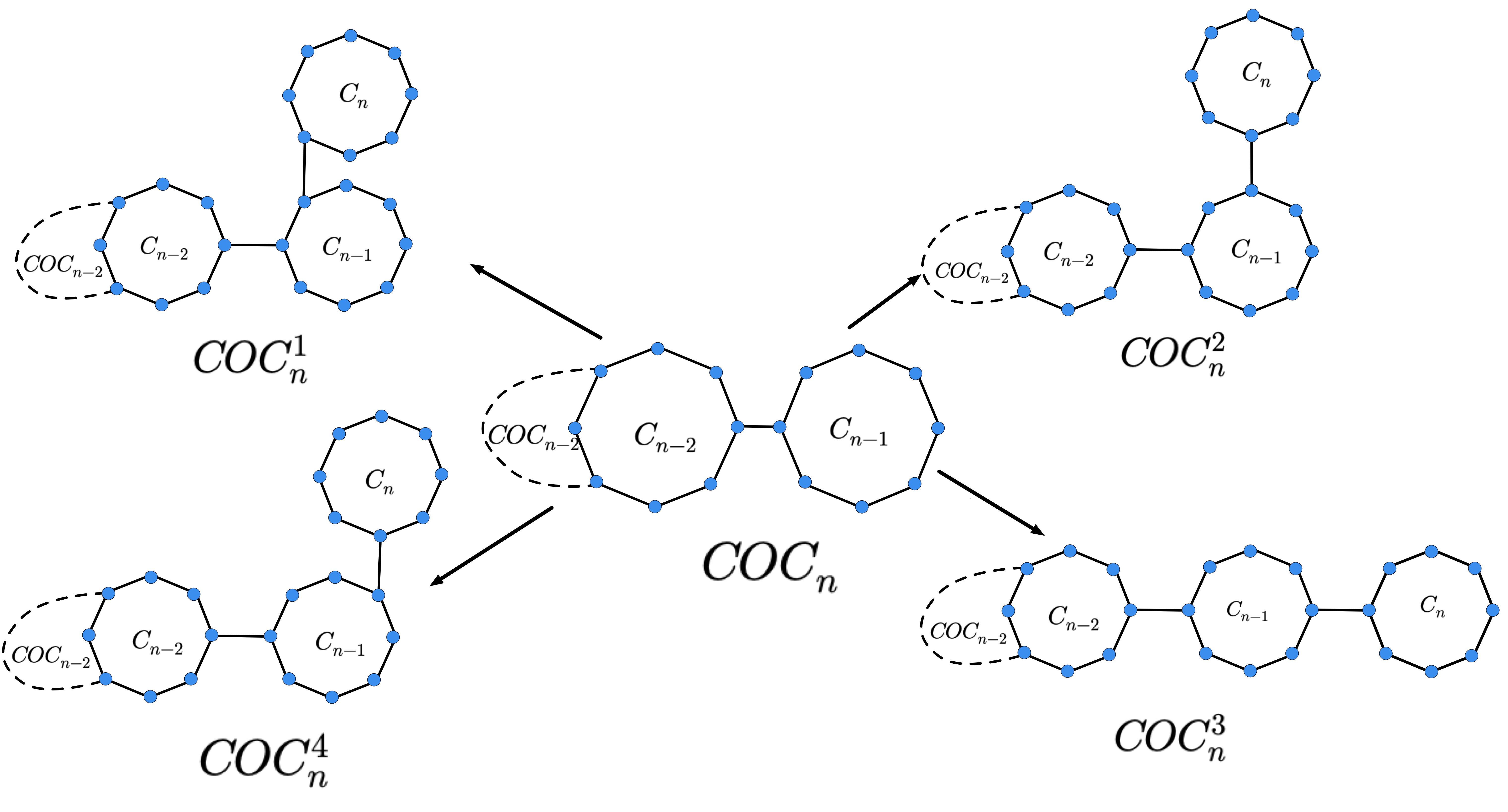}
    \caption{Four sorts of permutations in a random cyclooctane chain.}
    \end{figure}

The random $8$-polygonal chain $COC_n(n;\ p_1,\ p_2,\ p_3,\ 1-p_1-p_2-p_3)$ is a random cyclooctanechain with $n$ octagons.
Aromatic hydrocarbons and their derivatives have constantly attracted the interest of chemists\cite{WSXK,HJM}. According to the definition of a random $l$-polygon chain, there are four sorts of permutations in random cyclooctane chains (see Fig 6).
According to the calculation of the Sombor indices for a random polygonal chain presented in Section 3, the generalized formula of the Sombor indices of $COC_n$ is obtained by 
\begin{eqnarray*}
    SO_a(COC_n)
        &=&(2\sqrt{2}\left\lvert 2-a\right\rvert-2v_1) X+(5\sqrt{2} \left\lvert 2-a\right\rvert+\sqrt{2} \left\lvert 3-a\right\rvert +4v_1) n\\
        &+&\sqrt{2} \left\lvert 2-a\right\rvert -\sqrt{2} \left\lvert 3-a\right\rvert -4v_1 ,
           \end{eqnarray*}
where $v_1=\sqrt{2a^2-10a+13}$.

\begin{thm}
    The expected value and variance of the Sombor indices for $COC_n$ are obtained by
    \begin{eqnarray*}
         \mathbb{E} \Bigl(SO_a(COC_n)\Bigr)
            &=&\Bigl(\sqrt{2} ((5+p_1)\left\lvert 2-a\right\rvert+\sqrt{2} (1+p_1)\left\lvert 3-a\right\rvert  )+(4-2p_1)v_1\Bigr) n\\
            &-&(2\sqrt{2p_1+6} )\left\lvert 2-a\right\rvert -3\sqrt{2} \left\lvert 3-a\right\rvert -4p_1v_1 ,
    \end{eqnarray*}
    \begin{eqnarray*}
        \mathbb{V} \mathbf{a} \mathbf{r} \Bigl(SO_a(COC_n)\Bigr)
        =\Bigl(3v_1^2+2\left\lvert 2-a\right\rvert *\left\lvert 3-a\right\rvert -2\sqrt{2} (\left\lvert 2-a\right\rvert*\left\lvert 3-a\right\rvert  )\Bigr)(n-2)p_1(1-p_1),
    \end{eqnarray*}
        where $v_1=\sqrt{2a^2-10a+13}$.
            \end{thm}

\section{Asymptotic behaviors for $SO_a (G_n)$}
\ \ \ \ The distribution of the Sombor indices of random polygonal chains is presented in Section 3, see Corollaries $4.2-4.3$, respectively. In addition, the preceding section's applications to four specific polygonal chains were integrated. Therefore, in this part, asymptotic behavior of the Sombor indices within those random chains.

Under specific conditions, the normal distribution is a good approximation of the binomial distribution and can be used to calculate the probability of the binomial distribution. Since the probability thus obtained is only an approximation to the true probability value of the binomial distribution, this application of the normal distribution is known as the normal approximation to the binomial distribution.
\begin{thm}
    Let $X$ obeys bernoulli distribution. For $n \to \infty$,  $X$ asymptotically obeys normal distributions. One has 
    \begin{eqnarray*}
        \lim_{n \to \infty} \sup _{a\in \mathbb{R} } \left\lvert \mathbb{P} (\frac{X_n-np}{np(1-p) }\leqslant a ) -\int_{-\infty}^{a} \frac{1}{\sqrt{2\pi } } e^{\frac{t^2}{2} } \,dt\right\rvert=0.
      \end{eqnarray*}
\end{thm}

Theorem 6.1 leads to the following conclusion.

\begin{prop}
    Let $PC_n$, $GPC_n$, $P_n$, $PPC_n$, and $COC_n$ be a random polygonal chain, polyonino chain, pentachain, polyphenyl, and cyclooctane chain, respectively. 
    
    For each $x\in \mathbb{R},\ n \to \infty$, we have 
    \begin{eqnarray}
        \lim_{n \to \infty} \sup _{x\in \mathbb{R} } \left\lvert \mathbb{P} (\frac{SO_a(PC_n)-\mathbb{E} \Bigl(SO_a(PC_n)\Bigr)}{\sqrt{\mathbb{V} \mathbf{a} \mathbf{r}\Bigl(SO_a(PC_n)\Bigr)} }\leqslant x ) -\int_{-\infty}^{x} \frac{1}{\sqrt{2\pi } } e^{\frac{t^2}{2} } \,dt\right\rvert=0 ,
      \end{eqnarray}
      \begin{eqnarray}
        \lim_{n \to \infty} \sup _{x\in \mathbb{R} } \left\lvert \mathbb{P} (\frac{SO_a(GPC_n)-\mathbb{E} \Bigl(SO_a(GPC_n)\Bigr)}{\sqrt{\mathbb{V} \mathbf{a} \mathbf{r}\Bigl(SO_a(GPC_n)\Bigr)} }\leqslant x ) -\int_{-\infty}^{x} \frac{1}{\sqrt{2\pi } } e^{\frac{t^2}{2} } \,dt\right\rvert=0 ,
      \end{eqnarray}
      \begin{eqnarray}
        \lim_{n \to \infty} \sup _{x\in \mathbb{R} } \left\lvert \mathbb{P} (\frac{SO_a(P_n)-\mathbb{E} \Bigl(SO_a(P_n)\Bigr)}{\sqrt{\mathbb{V} \mathbf{a} \mathbf{r}\Bigl(SO_a(P_n)\Bigr)} }\leqslant x ) -\int_{-\infty}^{x} \frac{1}{\sqrt{2\pi } } e^{\frac{t^2}{2} } \,dt\right\rvert=0 ,
      \end{eqnarray}
      \begin{eqnarray}
        \lim_{n \to \infty} \sup _{x\in \mathbb{R} } \left\lvert \mathbb{P} (\frac{SO_a(PPC_n)-\mathbb{E} \Bigl(SO_a(PPC_n)\Bigr)}{\sqrt{\mathbb{V} \mathbf{a} \mathbf{r}\Bigl(SO_a(PPC_n)\Bigr)} }\leqslant x ) -\int_{-\infty}^{x} \frac{1}{\sqrt{2\pi } } e^{\frac{t^2}{2} } \,dt\right\rvert=0 ,
      \end{eqnarray}
      \begin{eqnarray}
        \lim_{n \to \infty} \sup _{x\in \mathbb{R} } \left\lvert \mathbb{P} (\frac{SO_a(COC_n)-\mathbb{E} \Bigl(SO_a(COC_n)\Bigr)}{\sqrt{\mathbb{V} \mathbf{a} \mathbf{r}\Bigl(SO_a(COC_n)\Bigr)} }\leqslant x ) -\int_{-\infty}^{x} \frac{1}{\sqrt{2\pi } } e^{\frac{t^2}{2} } \,dt\right\rvert=0 .
      \end{eqnarray}
\end{prop}

\begin{proof}
    By Corollary 4.2 and Corollary 4.3, we have 
    \begin{equation*}
        \frac{SO_a(PC_n)-\mathbb{E} \Bigl( SO_a(PC_n)\Bigr)} {\sqrt{\mathbb{V} \mathbf{a} \mathbf{r}\Bigl(SO_a(PC_n)\Bigr)}} =\frac{X-\mathbb{E} (X)}{\sqrt{\mathbb{V} \mathbf{a} \mathbf{r}(X)} },\ X\sim B(n-2,\ p_1).
    \end{equation*}

By Theorem 6.1 and $(4.6-4.19), (6.20)$ can be proved.

Similarly, $(6.21-6.24)$ are obtained by such proof process.
\end{proof}

According to Proposition 6.2, the distributions of the Sombor indices of the random polygonal chain networks are consistent with the asymptotic normal distributions, and when $n > 30$, $p_1$ is a constant between $0$ and $1$. The normal distributions can be regarded as an approximation to the distributions of the Sombor indices of the random chain networks in this paper,and please refer to Table 1 for the specific values.

\begin{table}[htbp]
    \setlength{\abovecaptionskip}{0.05cm}
     \centering \vspace{.3cm}
    \caption{Expected values and variances of Normal distribution for Sombor indices.}
    \begin{tabular}{cl}
      \hline
      $Indices$ & $Normal\ \ parameters$ \\[1.2ex]
      \hline
      $SO(PC_n)_{2k+1}$ & $\mu =\Bigl((5\sqrt{2}-2\sqrt{13} )p_1+4\sqrt{2}k-3\sqrt{2}+4\sqrt{13}\Bigr)n
      +(4\sqrt{13}-4\sqrt{2}  )p_1-8\sqrt{2}k+11\sqrt{2} $\\[1.3ex]
      $ \ \ $ &$\sigma ^2=(16\sqrt{26}+84 )p_1(1-p_1)(n-2)$ \\ [1.3ex]
      $SO_{red}(PC_n)_{2k+1}$ & $\mu =\Bigl((3\sqrt{2}-2\sqrt{5})p_1+2\sqrt{2}k+4\sqrt{5} -\sqrt{2} \Bigr)n+(4\sqrt{5}-2\sqrt{2})p_1-4\sqrt{2}k+4\sqrt{2}$ \\[1.3ex]
      $ \ \ $ & $\sigma ^2=(8\sqrt{10}+28 )(n-2)p_1(1-p_1)$ \\ [1.3ex]
      $SO_{avr}(PC_n)_{2k+1}$ & $\mu =M_1n+N_1 $\\ [1.3ex]
      $ \ \ $ & $\sigma ^2=\widetilde{\sigma } (n-2)p_1(1-p_1) $ \\ [1.3ex]
      $SO(PC_n)_{2k}$ & $\mu =\Bigl((2p_1(\sqrt{2}-\sqrt{13} )+4\sqrt{2}k-4\sqrt{13}-3\sqrt{2}\Bigr)n+4p_1(\sqrt{26}-\sqrt{2} ) +8\sqrt{2}k+5\sqrt{2}-8\sqrt{26}$ \\ [1.3ex]
      $ \ \ $    & $\sigma ^2=(84-16\sqrt{26} )p_1(1-p_1)(n-2)$ \\ [1.3ex]
      $SO_{red}(PC_n)_{2k}$ & $\mu=\Bigl(\sqrt{2} (p_1+2k-1)+(4-2p_1)\sqrt{5} \Bigr)n-2\sqrt{2} p_1+4\sqrt{10} p_1-4\sqrt{2} k-8\sqrt{10} +\sqrt{2} $\\ [1.3ex]
      $\ \ $ & $\sigma ^2=(28-16\sqrt{10} )(n-2)p_1(1-p_1)$\\ [1.3ex]
      $SO_{avr}(PC_n)_{2k}$ & $\mu=\frac{-2\sqrt{2}np_1+2\sqrt{2}p_1+8\sqrt{2}n-6\sqrt{2}}{k}+\frac{4np_1-8n-4}{nk} \mu _2$\\ [2.5ex]
      $\ \ $ & $\sigma ^2=\frac{-80n^2k^2-8n^2k+8nk+16n^2-32n+16}{n^2k^2}+\frac{-8\sqrt{2} (n-1)\mu _2 }{n^2k^2}$ \\ [2.5ex]
      \hline
    \end{tabular}
    \end{table}

\section{Conclusion}
\ \ \ \ In this paper, a method for calculating the distributions of the Sombor indices of a random polygonal chain has been established. The expected values and variances of the Sombor indices of a random polygonal chain have been calculated. As an application, we also have obtained the expected values and variances of the Sombor indices for polyonino chain, pentachain, polyphenyl chain, and cyclooctane chain.
Based on the central limit theorem, it is also discovered from a probabilistic perspective that since the end connections of random chains obey a binomial distribution, when the number $n$ of polygons connected by any chain tends to infinity, is the Sombor indices of any chain at that point from a normal distribution. Sombor indices can help to predict the physicochemical properties of many molecules. We have compared the Sombor indices with some existing topological indices and will find that sometimes the Sombor indices show better predictive power than existing indices, which also provide new ideas in our future research discussions.

\section*{Declartion of competing interest}
\ \ \ \ We confirm that these results have neither published elsewhere nor are under consideration. The authors have no conflict of interest to disclose.

\section*{Data availability statements}
\ \ \ \ The data that support the finding of this study are availability within the article. All relevant data are also availability from the corresponding author upon reasonable request.



\end{document}